\documentclass[oneside,reqno]{amsart}
\usepackage{hyperref,amssymb,amsmath,amsthm,color}
\usepackage{geometry,mdframed}
\hypersetup{
	colorlinks = true,
	linkcolor = blue,
	citecolor = blue
}
\usepackage{tikz-cd}
\definecolor{theoremback}{rgb}{0.8235294, 0.8627451, 0.9137255}
\definecolor{exampletitle}{rgb}{0.6705882, 0.7568627, 0.8705882}
\usepackage{fancyhdr}
\newtheoremstyle{nthm}
{3pt}
{3pt}
{\itshape}
{0em}
{\bfseries}
{.}
{.5em}
{}

\newtheoremstyle{ndef}
{3pt}
{3pt}
{}
{0em}
{\bfseries}
{.}
{.5em}
{}

\newtheoremstyle{nrem}
{3pt}
{3pt}
{}
{0em}
{\itshape}
{.}
{.5em}
{}

\swapnumbers
\theoremstyle{nthm}
\newmdtheoremenv[linecolor=theoremback,linewidth=2,backgroundcolor=theoremback]{thm}[subsection]{Theorem}
\newtheorem{lem}[subsection]{Lemma}
\newtheorem{prop}[subsection]{Proposition}
\newtheorem{cor}[subsection]{Corollary}

\theoremstyle{ndef}

\theoremstyle{nrem}

\newtheorem{exmp}[subsection]{Example}

\newcommand{\Z}{\mathbb{Z}}

	\DeclareMathOperator{\Hom}{Hom}

\newcommand{\df}[1]{{\bfseries \emph{#1}}}

\newsavebox{\fmbox}

\DeclareMathOperator{\tr}{tr}
\DeclareMathOperator{\disc}{disc}

\newcommand{\theoremOne}{Let $p$ be a prime and $k \geq 1$ be an integer. The number of monic separable polynomials of degree $d$ with $d \geq 2$ in $\Z/p^k[x]$ is $p^{kd-1}(p-1) = \phi(p^{kd})$. Equivalently, the proportion of monic polynomials of degree $d$ with $d\geq 2$ that are separable in $\Z/p^k[x]$ is $\left( 1-p^{-1} \right)$.}
\newcommand{\theoremTwo}{Let $n$ be an integer with $|n| > 1$ and let $n = p_1^{k_1}\cdots p_m^{k_m}$ be the prime factorisation of $n$. Then, the number of monic separable polynomials of degree $d$ with $d\geq 2$ in the ring $\Z/n[x]$ is $\phi(n^d)$. Equivalently, the proportion of monic polynomials of degree $d$ with $d \geq 2$ that are separable in $\Z/n$ is equal to
    \begin{align*}
        \prod_{i=1}^m \left(1-p_i^{-1}\right).
    \end{align*}}
\newcommand{\theoremThree}{ Let $n$ be a positive integer with prime factorisation $n = p_1^{k_1}\cdots p_m^{k_m}$ and let $d\geq 1$. Then the number of separable polynomials $f$ with $\deg(f) \leq d$ in $\Z/n[x]$ is
    \begin{align*}
        \phi(n)n^d\prod_i(1+p_i^{-d}).
    \end{align*}}

\begin{document}
\title{ Counting Separable Polynomials in $\Z/n[x]$ }
\author{Jason K.C. Polak}
\thanks{This research was made possible by ARC Grant DP150103525.}
\date{\today}
\begin{abstract}
    For a commutative ring $R$, a polynomial $f\in R[x]$ is called separable if $R[x]/f$ is a separable $R$-algebra. We derive formulae for the number of separable polynomials when $R = \Z/n$, extending a result of L. Carlitz. For instance, we show that the number of separable polynomials in $\Z/n[x]$ that are separable is $\phi(n)n^d\prod_i(1-p_i^{-d})$ where $n = \prod p_i^{k_i}$ is the prime factorisation of $n$ and $\phi$ is Euler's totient function.
\end{abstract}
\subjclass[2010]{Primary: 16H05. Secondary: 13B25,13M10}
\address{School of Mathematics and Statistics\\The University of Melbourne\\Parkville, Victoria 3010\\Australia}
\email{jpolak@jpolak.org}
\keywords{Separable algebras, separable polynomials}
\maketitle

\section{Introduction}
Suppose $a,b,c$ are independently, uniformly randomly chosen elements of $\Z/11$. What is the probability that the element $a^{2}b^{2} - 4a^{3} c - 4b^{3} + 18abc - 27c^{2}$ is nonzero in $\Z/11$? Answer: $10/11$. This peculiar fact follows from a theorem of Carlitz~\cite[\S6]{Carlitz1932}, who proved that the number of monic separable polynomials in $\Z/p[x]$ of degree $d$ where $d\geq 2$ is $p^d - p^{d-1}$. Our aim is to extend his result to separable polynomials in $\Z/n[x]$. Now, most people are familiar with separable polynomials over fields, but just \emph{what is} a separable polynomial over an arbitrary commutative ring? To understand separable polynomials, we will have to first look at separable algebras.

Let $R$ be a commutative ring. If $A$ is an $R$-algebra, we define $A^{\rm op}$ to be the ring with the same underlying abelian group as $A$ and whose multiplication is given by $(a,b)\mapsto ba$. Then $A$ can be made into a left $A\otimes_R A^{\rm op}$-module via the action $(a\otimes a')b = aba'$. An $R$-algebra $A$ is called \df{separable} if $A$ is projective as an $A\otimes_R A^{\mathrm{op}}$-module, the basic theory of which is contained in \cite{AuslanderGoldm1960}. Examples include separable field extensions, full matrix rings over a commutative ring $R$, and group rings $k[G]$ when $k$ is a field and $G$ is a finite group whose order is invertible in $k$. On the other hand $
\Z[\sqrt{5}]$ is not a separable $\Z$-algebra.

A polynomial $f\in R[x]$ is called \df{separable} if $R[x]/f$ is separable as an $R$-algebra. A monic polynomial is separable if and only if the ideal $(f,f')$ is all of $R[x]$ \cite[\S1.4, Proposition 1.1]{Magid2014}, and so for fields coincides with the usual definition. For example, $x-a$ for $a\in R$ is always separable. To state our results, recall Euler's \df{totient function}: for a positive integer $n$, the number $\phi(n)$ is the number of elements of the set $\{1,2,\dots,n\}$ relatively prime to $n$. In other words, $\phi(n) = |(\Z/n)^\times|$. For example, $\phi(p^k) = p^k - p^{k-1}$. Our first theorem is on the number of monic separable polynomials in $\Z/p^k[x]$:

\begin{thm}\label{thm:mainthm}\theoremOne\end{thm}

When $k=1$ this result is Carlitz's theorem. For example, when $k = 1$ and $d= 2$, this result is easily computable. There are $p^2$ monic quadratic polynomials. Since $\Z/p$ is a perfect field, every irreducible quadratic is separable. Therefore, the only quadratic polynomials that are not separable are of the form $(x-a)^2$ for $a\in \Z/p$. Therefore, there are $p^2 - p$ separable quadratics. We note that in general, one must be careful with factorisation when $k > 1$ since then $\Z/p^k[x]$ is not a unique factorisation domain: for example, in $\Z/4[x]$, we have $x^2 = (x + 2)^2$.

\begin{exmp}If $k = 1$, then $\Z/p^k = \Z/p$ is a field, and every irreducible polynomial is also separable. This is not true if $k > 1$. For example, in $\Z/4[x]$, the polynomial $x^2 + 1$ is irreducible, but not separable. On the other hand, $x^2 + x + 1$ is separable and irreducible.\end{exmp}

Next, we consider the general case of $\Z/n$.

\begin{thm}\theoremTwo
\end{thm}

Since there exists separable polynomials in $\Z/n[x]$ that are not monic and whose leading coefficient is not a unit, our ultimate aim is to count all separable polynomials in $\Z/n[x]$:
\begin{thm}
    \theoremThree
\end{thm}

\section{Monic Separable Polynomials in $\Z/n[x]$}

Let $R$ be a commutative ring. If $A$ is a finitely generated projective $R$-module, then there exists elements $f_1,\dots,f_n\in \Hom_R(A,R)$ and $x_1,\dots,x_n\in A$ such that for all $x\in A$,
\begin{align*}
    x = \sum_{i=1}^n f_i(x)x_i.
\end{align*}
The elements $\{ f_i,x_i\}_{i=1}^n$ are called a \df{dual basis} for $A$. If $A$ is additionally an $R$-algebra, one can define the \df{trace map} to be
\begin{align*}
    \tr:A&\longrightarrow R\\
    x&\longmapsto \sum_{i=1}^n f_i(xx_i) 
\end{align*}
It is easy to see that the trace map is independent of the chosen dual basis. If $f$ is a monic polynomial in $R[x]$, then the algebra $R[x]/f$ is a finitely generated free $R$-module. One possible dual basis for $R[x]/f$ is $x_i = x^i$ with $f_i(g)$ being the coefficient of $x^i$ in $g$, where $i=1,\dots,n-1$ with $n=\deg(f)$. We can use part of Theorem 4.4 in Chapter III of~\cite{DeMeyerIngraham1971} to decide when a polynomial $f\in R[x]$ is separable:
\begin{prop}\label{thm:sepdetect}
    Let $R$ be a commutative ring with no nontrivial idempotents and let $f\in R[x]$ be a degree $n$ monic polynomial. Let $A$ be the matrix whose $(i+1,j+1)$-entry is $\tr(x^{i+j})$. Then $f$ is separable if and only if $\disc(f):= \det(A)\in R$ is a unit.
\end{prop}
\begin{exmp}
    Consider $f = x^2 + ax + b\in \Z/p[x]$. Then the matrix $A$ in Theorem~\ref{thm:sepdetect} is
    \begin{align*}
        A=
        \begin{pmatrix}
            2 & -a \\
            -a & a^2 - 2b
        \end{pmatrix}
    \end{align*}
    Its determinant is the familiar $a^2 - 4b$. If $f = x^3 + ax^2 + bx + c$ then
    \begin{align*}
        A = 
        \begin{pmatrix}
            3 & -a & a^2 - 2b\\
            -a & a^2 - 2b & -a^3 + 3ab - 3c\\
            a^2 - 2b & -a^3 + 3ab - 3c & a^4 - 4a^2b + 4ac + 2b^2
        \end{pmatrix}
    \end{align*}
    and its determinant is the less familiar $a^{2}b^{2} - 4a^{3} c - 4b^{3} + 18abc - 27c^{2}$. This explains the relation of separable polynomials to the opening paragraph's bizarre question.
\end{exmp}

We now prove:

\begin{thm}\theoremOne\end{thm}

\begin{proof}
    From Proposition~\ref{thm:sepdetect}, $f\in \Z/p^k[x]$ is separable if and only if its discriminant $\disc(f)$ is invertible in $\Z/p^k$. Since $\disc(f)$ is obtained from the coefficients of $f$ through basic arithmetic operations of addition and multiplication, we see that $f$ is separable if and only if its image in $\Z/p^k[x]/p\Z/p^k[x]\cong \Z/p[x]$ is separable. Hence we have reduced the problem to Carlitz's theorem.
\end{proof}

Now that we have determined the number of separable polynomials in $\Z/p^k[x]$, we move on to the general case of $\Z/n$ for any integer $n$ with $|n| > 1$.  We will need the following result.
\begin{prop}[{\cite[Proposition II.2.1.13]{DeMeyerIngraham1971}}]\label{thm:prodsep}
    Let $R_1$ and $R_2$ be commutative rings and let $A_i$ be a commutative $R_i$ algebra for $i=1,2$. Then $A_1\times A_2$ is a separable $R_1\times R_2$-algebra if and only if $A_i$ is a separable $R_i$ algebra for $i=1,2$.
\end{prop}

\begin{thm}
\theoremTwo
\end{thm}
\begin{proof}
    Factor $n=p_1^{k_1}p_2^{k_2}\cdots p_m^{k_m}$ where the $p_i$ are the prime factors of $n$ so that $\Z/n \cong \Z/p_1^{k_1}\times\cdots\times\Z/p_m^{k_m}$. Then we have
    \begin{align}\label{eqn:znDecomp}
    \Z/n[x] \cong \Z/p_1^{k_1}[x]\times\cdots\times\Z/p_m^{k_m}[x].
\end{align}
An element $f\in \Z/n[x]$ corresponds to an element $(f_1,\dots,f_m)\in \Z/p_1^{k_1}[x]\times\cdots\times\Z/p_m^{k_m}[x]$. From Proposition~\ref{thm:prodsep}, we see that $f$ is separable if and only if $f_i$ is a separable polynomial in $\Z/p_i^{k_i}$. Therefore, the number of monic polynomials over $\Z/n$ that are separable of degree $d$ is equal to the number of tuples $(f_1,\dots,f_m)$ such that $f_i$ is a separable monic polynomial over $\Z/p_i^{k_i}$ for each $i$ and $\deg(f_i) = d$. The result now follows from Theorem~\ref{thm:mainthm} and the fact that $\phi$ is multiplicative.
\end{proof}

\begin{exmp}
    For $n = 614889782588491410$ (the product of the first fifteen primes) the proportion of monic polynomials over $\Z/n$ that are separable is $1605264998400/11573306655157$, or about $0.138704092635850$. The formula shows that as the number of prime factors of $n$ increases to infinity, the proportion of separable polynomials goes to zero.
\end{exmp}

\section{Arbitrary Polynomials and Separability}

In the case of fields, it suffices to look at monic polynomials since one can always multiply such a polynomial by a unit to make it monic, and this does not change the ideal it generates. For general rings, this is not so. And, it is clear from the isomorphism in~\eqref{eqn:znDecomp} that there are many polynomials that are separable are not monic and whose leading coefficient is not invertible. 
\begin{exmp}
   In the ring $\Z/6[x]$, the polynomial $f = 3x^2 +  x + 5$ is separable and irreducible, but its leading coefficient is not a unit in $\Z/6$.
\end{exmp}
In this section we calculate the number of separable polynomials of at most degree $d$ where $d\geq 1$, and whose leading coefficient is arbitrary. As before, this result depends on the result for polynomials in $\Z/p^k[x]$. We have already observed that a monic polynomial is separable in $\Z/p^k[x]$ if and only if its reduction modulo $p$ is reducible in $\Z/p[x]$. To handle arbitrary polynomials, we use the following theorem.
\begin{prop}[{\cite[II.7.1]{DeMeyerIngraham1971}}]Let $R$ be a commutative ring. For a finitely generated $R$-algebra $A$, the following are equivalent:
    \begin{enumerate}
        \item $A$ is a separable $R$-algebra,
        \item $A_m$ is a separable $R_m$ algebra for every maximal ideal $m$ of $R$, and
        \item $A/mA$ is a separable $R/m$-algebra for every maximal ideal $m$ of $R.$
    \end{enumerate}
\end{prop}
Since $\Z/p^k$ is a local ring with unique maximal ideal $(p)$:
\begin{cor}\label{thm:reductionSeparable}
    A polynomial $f\in \Z/p^k[x]$ is separable if and only if its reduction in $\Z/p[x]$ modulo $p$ is separable.
\end{cor}

\begin{exmp}\label{exmp:degzerodegone}
    Let $a\in \Z/p^k\subseteq\Z/p^k[x]$ be a constant polynomial. In $\Z/p[x]$, the zero polynomial is not separable since $\Z/p[x]$ is not a separable $\Z/p$-algebra; indeed, a separable algebra over field must be finite-dimensional over that field \cite[II.2.2.1]{DeMeyerIngraham1971}. Therefore $a$ is separable if and only if $a$ is a unit in $\Z/p^k$. Thus, there are $\phi(p^k)$ separable polynomials in $\Z/p^k[x]$ of degree zero. 

    We can proceed inductively to calculate the number of separable polynomials of degree one. They are one of two types, according to Corollary~\ref{thm:reductionSeparable}:
    \begin{enumerate}
        \item $ux + b$ where $u$ is a unit.
        \item $ux + b$ where $u\not=0$ is not a unit and $b$ is a unit.
    \end{enumerate}
    In the first case, we already know there are $p^k$ monic separable linear polynomials, so there are $\phi(p^k)p^k$ polynomials whose leading coefficient is a unit. In the second case there are $(p^k - \phi(p^k)-1)\phi(p^k)$ polynomials of the form $ux + b$ where $u\not=0$ is not a unit but $b$ is a unit. Adding these two together, we see that there $\phi(p^k)(2p^k - \phi(p^k) - 1) = \phi(p^k)(p^k + p^{k-1} - 1)$ linear separable polynomials in $\Z/p^k[x]$, and hence $\phi(p^k)(p^k + p^{k+1})$ separable polynomials of degree at most one.
\end{exmp}

Now we will derive a formula for the number of separable polynomials with arbitrary leading coefficient and degree at most $d$. But first, we will need the following elementary geometric sum:
\begin{lem}\label{thm:geomseries}Let $\beta = p^k$ and $\lambda = p^{k-1}$. Then
    \begin{align*}\phi(\beta^d) + \lambda\phi(\beta^{d-1}) + \cdots + \lambda^{d-2}\phi(\beta^2)=p^{(k-1)d + 1}(p^{d-1}-1).
    \end{align*}
\end{lem}
\begin{proof}We sum a geometric series
    \begin{align*}
        \phi(\beta^d) + \lambda\phi(\beta^{d-1}) + \cdots + \lambda^{d-2}\phi(\beta^2) &= p^{kd-1}(p-1)\left[ 1 + \frac{\lambda}{p^k} + \left( \frac{\lambda}{p^k} \right)^2 + \cdots + \left( \frac{\lambda}{p^k} \right)^{d-2} \right]\\
        &= p^{kd-1}(p-1)\left[ 1 + p^{-1} + p^{-2} + \cdots + p^{-(d-2)}\right]\\
        &= p^{kd-1}(p-1)\frac{1 - p^{d-1}}{(1-p)p^{d-2}}\\
        &= p^{kd-1}\frac{p^{d-1}-1}{p^{d-2}}\\
        &= p^{(k-1)d + 1}(p^{d-1}-1).
    \end{align*}
\end{proof}

\begin{thm}\label{thm:primepowerlessthan}
    For $d\geq 1$, number of separable polynomials $f$ such that $\deg(f)\leq d$ in $\Z/p^k[x]$ is
    \begin{align*}
        \phi(p^k)p^{(k-1)d}(p^d+1) = \phi(p^k)p^{kd}(1+p^{-d})
    \end{align*}
\end{thm}
\begin{proof}
    Let $a_d$ be the number of separable polynomials $f$ in $\Z/p^k[x]$ with arbitrary leading coefficient and such that $\deg(f)\leq d$. We follow the calculation method in Example~\ref{exmp:degzerodegone} A separable polynomial of degree $d$ must be either have unit leading coefficient, or else it must be of the form $ux^d + g$ where $u\not=0$ is not a unit and $g$ is a separable polynomial of $\deg(g) < d$. We have already shown that the number of monic polynomials of degree $d\geq 2$ in $\Z/p^k[x]$ that are also separable is $\phi(p^{kd})$. Therefore, the number of separable polynomials with unit leading coefficient and of degree exactly $d$ for $d\geq 2$ is
    \begin{align*}
        \phi(p^{kd})\phi(p^k).
    \end{align*}
     With our notation, the number $a_d - a_{d-1}$ is the number of separable polynomials of degree exactly $d$, and our reasoning shows that we have the recurrence relation
\begin{align*}
    a_d - a_{d-1} = \phi(p^{kd})\phi(p^k) + (p^k - \phi(p^k) - 1)a_{d-1},
\end{align*}
which simplifies to
\begin{align*}
    a_d &= \phi(p^{kd})\phi(p^k) + (p^k - \phi(p^k))a_{d-1}\\
    &= \phi(p^{kd})\phi(p^k) + p^{k-1}a_{d-1}
\end{align*}
as long as $d\geq 2$. To simplify notation for intermediate computations, let us set $\beta = p^k$ and $\lambda = p^{k-1}$. Then $a_d = \phi(\beta^d)\phi(\beta) + \lambda a_{d-1}$. It is easy to see that
\begin{align*}
    a_d = \phi(\beta)\left[\phi(\beta^d) + \lambda\phi(\beta^{d-1}) + \cdots + \lambda^{d-2}\phi(\beta^2)\right] + \lambda^{d-1}a_1.
\end{align*}
Example~\ref{exmp:degzerodegone} shows that $a_1 = \phi(p^k)(p^k + p^{k-1})$ so that $\lambda^{d-1}a_1 = \phi(p^k)p^{(k-1)d}(p+1)$. Now, using Lemma~\ref{thm:geomseries}, we get
\begin{align*}
    a_d &= \phi(p^k)\left[ p^{(k-1)d + 1}(p^{d-1} - 1) \right] + \phi(p^k)p^{(k-1)d}(p+1)\\
    &= \phi(p^k)p^{(k-1)d}\left[ p(p^{d-1}-1) + p + 1 \right]\\
    &= \phi(p^k)p^{(k-1)d}(p^d+1).
\end{align*}
This takes care of $d\geq 2$. Putting $d=1$ into this last line shows that it is equal to $a_1$, so the formula is also valid for $d=1$.
\end{proof}
\begin{thm}\theoremThree
\end{thm}
\begin{proof}
    This follows from Theorem~\ref{thm:primepowerlessthan} and the fact that Euler's totient function $\phi$ is a multiplicative arithmetic function in the sense that $\phi(mn) = \phi(m)\phi(n)$ whenever $m$ and $n$ are relatively prime.
\end{proof}

\begin{exmp}
    There are $65028096$ separable polynomials in $\Z/120[x]$ of degree at most three. There are $1888$ separable polynomials of degree exactly two in $\Z/15[x]$. \end{exmp}

    \bibliographystyle{alpha}
    \bibliography{biblio.bib}

\end{document}